\documentclass[11pt,reqno,twoside]{amsart}
\usepackage{graphicx}
\usepackage{amssymb}
\usepackage{epstopdf}
\usepackage[asymmetric,top=3.5cm,bottom=4.3cm,left=3.1cm,right=3.1cm]{geometry}
\usepackage{graphicx}
\usepackage{bm}

\usepackage{booktabs} 
\usepackage{array} 
\usepackage{paralist} 
\usepackage{verbatim} 
\usepackage{subfig} 
\usepackage{tabularx}
\usepackage{amsmath,amsfonts,amsthm,mathrsfs,amssymb,cite}
\usepackage[usenames]{color}
\usepackage{ulem}

\newtheorem{thm}{Theorem}[section]

\newtheorem{lem}{Lemma}[section]
\newtheorem{prop}{Proposition}[section]
\theoremstyle{definition}

\theoremstyle{remark}

\newtheorem{rem}{Remark}[section]
\numberwithin{equation}{section}

\newcommand{\Kcal}{\mathcal{K}}
\newcommand{\Scal}{\mathcal{S}}
\newcommand{\Dcal}{\mathcal{D}}
\newcommand{\Acal}{\mathcal{A}}
\newcommand{\Rcal}{\mathcal{R}}

\newcommand{\bu}{\mathbf{u}}

\newcommand{\bff}{\mathbf{f}}
\newcommand{\bF}{\mathbf{F}}
\newcommand{\be}{\mathbf{e}}
\newcommand{\bx}{\mathbf{x}}
\newcommand{\by}{\mathbf{y}}
\newcommand{\bz}{\mathbf{z}}

\newcommand{\bGamma}{\mathbf{\Gamma}}
\newcommand{\bpsi}{\bm{\psi}}
\newcommand{\bvarphi}{\bm{\varphi}}
\newcommand{\bphi}{\bm{\phi}}
\newcommand{\balpha}{\bm{\alpha}}
\newcommand{\bbeta}{\bm{\beta}}
\newcommand{\bnu}{\bm{\nu}}

\newcommand{\Bx}{\mathbf{x}}
\newcommand{\By}{\mathbf{y}}

\newcommand{\tdx}{\tilde{\Bx}}
\newcommand{\tdy}{\tilde{\By}}

\newcommand{\RR}{\mathbb{R}}

\newcommand{\eqnref}[1]{(\ref {#1})}

\newcommand{\p}{\partial}

\newcommand{\beq}{\begin{equation}}
\newcommand{\eeq}{\end{equation}}

\title[Polariton resonance in elastic system]{Analysis of surface polariton resonance for nanoparticles in elastic system}

\author{Youjun Deng}
\address{School of Mathematics and Statistics, Central South University, Changsha, Hunan, P. R. China.}
\email{youjundeng@csu.edu.cn; dengyijun\_001@163.com}

\author{Hongjie Li}
\address{Department of Mathematics, Hong Kong Baptist University, Kowloon Tong, Hong Kong SAR, P. R. China.}
\email{hongjie$_-$li@yeah.net}

\author{Hongyu Liu}
\address{Department of Mathematics, Hong Kong Baptist University, Kowloon Tong, Hong Kong SAR, P. R. China.}
\email{hongyu.liuip@gmail.com; hongyuliu@hkbu.edu.hk}

\begin{document}
\maketitle

\begin{abstract}

This paper is concerned with the analysis of surface polariton resonance for nanoparticles in linear elasticity. With the presence of nanoparticles, we first derive the perturbed displacement field associated to a given elastic source field. It is shown that the leading-order term of the perturbed elastic wave field is determined by the Neumann-Poinc\'are operator associated to the Lam\'e system. By analyzing the spectral properties of the aforesaid Neumann-Poinc\'are operator, we study the polariton resonance for the elastic system. The results may find applications in elastic wave imaging.

\medskip

\medskip



\end{abstract}

\section{Introduction}
Recently, the mathematical study of surface plasmon resonances of nanoparticles have received considerable attention in the literature for their various applications including biomedicine imaging \cite{ADM, APRY18, JLEE06}, near-field microscopy \cite{BL2002,BLBCL2006}, molecular recognition \cite{HF04}, and heat therapeutic applications \cite{ARRu18, BGQu10, FDLi15}. Plasmonic materials are artificially engineered metamaterials that exhibit negative or left-handed optical properties. V. G. Veselago was the first to theoretically investigate the electrodynamics of substances with simultaneously negative permittivity $\varepsilon$ and magnetic permeability $\mu$ \cite{Ves}. In \cite{JP}, J. B. Pendry extended the work by V. G. Veselago and gave the theoretical construction of a perfect lens. Plasmonic materials have also been used to create invisibility cloaking \cite{Alu,ACKL13, ACKL14, AKLi16, KLO, KLSW14,La,LLLi15,NMMB07, MiNi06}. Negative refractive indexed materials have been fabricated in the lab \cite{Le}. Plasmonic materials are metals or metal-like materials, and at the resonance frequencies, the incident light couples with the surface plasmons to create a significant field enhancement around the nanoparticles which can be much shorter in wavelength and known as surface plasmon polaritons.

In this paper, we consider elastic scattering from nanoparticles with possibly negative material parameters. In metamaterials, exotic elastic materials with negative material parameters have been artificially engineered; see e.g. \cite{KM} and \cite{LLBW}. The so-called anomalous localized resonance and its associated cloaking effects have been intensively investigated for certain elastic polariton structures of the core-shell form within spherical and elliptical geometries \cite{AKKY,AKan16,AKM1,DLLi18,LLi16,LLi17,LLL2}, which are closely related to our current study. In this paper, we are mainly concerned with the analysis of surface polariton resonance for nanoparticles in linear elasticity. The surface plasmon resonance for nanoparticles in electromagnetism and its application in electromagnetic imaging were recently studied in \cite{ADM} by one of the authors of the present paper. Inspired by the study in \cite{ADM}, we make use of the layer potential techniques to reduce the elastic scattering problem into a system of integral equations. Indeed, the kind of techniques has been extensively used in the aforementioned works on plasmon resonances in optics and elastic cloaking. The resonance frequencies are closely related to the spectra of the involved integral operators, referred to as the Neumann-Poincar\'e-type operators. However, the spectral properties of the  Neumann-Poincar\'e operators in elasticity are much more complicated and delicate to analyze than those for the optical problems; see e.g. \cite{DLLi18,AKM1}. For the elastic scattering from nanoparticles, by performing asymptotic analysis, we derive that the leading-order term is determined by the Neumann-Poincar\'e operator associated to the elastic system. Then we show the polariton resonance by analyzing the spectral properties of the aforementioned Neumann-Poincar\'e operator. Similar to \cite{ADM}, the corresponding resonance result may find applications in elastic wave imaging.

The remaining part of this section is devoted to the description of the mathematical setup and the elastic scattering problem for our study. Let $D$ denote the nanoparticles that we are concerned with. It is assumed that $D$ is of the form
\begin{equation}\label{eq:form1}
D=\delta B + \bz,
\end{equation}
where $\delta\in\mathbb{R}_+$, $\bz\in\mathbb{R}^3$ and $B$ is a bounded domain containing the origin with a $C^\infty$ smooth boundary $\partial B$ and a connected complement $\mathbb{R}^3\backslash\overline{B}$. The elastic medium configuration in $\mathbb{R}^3$ is characterized by the Lam\'e parameters $\tilde{\lambda}$ and $\tilde{\mu}$, which are given as follows
\begin{equation}\label{eq:def01}
  (\tilde{\lambda},\tilde{\mu})= A(\bx)(\lambda,\mu),
\end{equation}
where
\begin{equation}\label{eq:def02}
A(\bx)=
  \left\{
    \begin{array}{ll}
      c, & \bx\in D,\medskip \\
      1, & \bx\in \mathbb{R}^3\backslash \overline{D},
    \end{array}
  \right.
\end{equation}
with $c\in\mathbb{C}$ and $\Im c \geq 0$, and $\lambda, \mu$ two constants satisfying the following strong convexity condition,
\begin{equation}\label{eq:convexity condition}
  \mu>0 \quad \mbox{and} \quad 3\lambda + 2\mu>0.
\end{equation}
It is noted that $B$ may have multiple connected components, corresponding to the case that there are multiple nanoparticles. Associated with the elastic medium configuration described above in $\mathbb{R}^3$, the elastic wave scattering is governed by the following Lam\'e system,
\begin{equation}\label{eq:general}
  \mathcal{L}_{\tilde{\lambda},\tilde{\mu}} \mathbf{u} + \omega^2 \mathbf{u}=\mathbf{f} \quad \mbox{in} \quad \mathbb{R}^3,
\end{equation}
where $\mathbf{u}\in H_{loc}^1(\mathbb{R}^3)^3$ and the the Lam\'e operator is defined by
\[
 \mathcal{L}_{\tilde{\lambda},\tilde{\mu}} \mathbf{u}:=\tilde{\mu}\Delta\mathbf{u}(\mathbf{x})+(\tilde{\lambda}+\tilde{\mu})\nabla\nabla\cdot\mathbf{u}(\mathbf{x}).
\]
Here we suppose that the source term $\mathbf{f}\in H^{-1}(\mathbb{R}^3)^3$ is compactly supported in $\RR^3\setminus\overline{D}$. To complete the description of the elastic system, we impose the following radiating condition for $\bu(\bx)$ in \eqref{eq:general} as $|\mathbf{x}|\rightarrow+\infty$,
\begin{equation}\label{eq:radiating}
\begin{split}
(\nabla\times\nabla\times\mathbf{u})(\bx)\times\frac{\bx}{|\bx|}-\mathrm{i}k_T\nabla\times\mathbf{u}(\bx)=&\mathcal{O}(|\bx|^{-2}),\\
\frac{\bx}{|\bx|}\cdot[\nabla(\nabla\cdot\mathbf{u})](\bx)-\mathrm{i}k_L\nabla\mathbf{u}(\bx)=&\mathcal{O}(|\bx|^{-2}),
\end{split}
\end{equation}
where $k_T:=\omega/\sqrt{\mu}$, $k_L:=\omega/\sqrt{\lambda+2\mu}$ and $\mathrm{i}:=\sqrt{-1}$ is the imaginary unit.

The rest of the paper is organized as follows. In Section 2, we first present the layer potential techniques for the elastic scattering problem and obtain the integral representation of the solution to the Lam\'e system. Then we derive the asymptotic expansion of the elastic wave field with respect to the size parameter of the nanoparticle. Section 3 is devoted to the resolvent analysis of the Neumann-Poicar\'e operator derived in Section 2 and the mathematical analysis of the surface polariton resonance in the elastic system.

\section{Integral reformulation and asymptotic analysis}

\subsection{Layer potential techniques}

In this subsection, we present some preliminary knowledge on layer potential techniques for the Lam\'e system for our subsequent use.
To begin with, we introduce the Kupradze matrix $\mathbf{\Gamma}^{\omega}$, which is the fundamental solution to the PDO (partial differential operator) $\mathcal{L}_{\lambda,\mu} + \omega^2$ as follows (cf. \cite{ABGK}),
\begin{equation}\label{eq:expression_gamma}
  \mathbf{\Gamma}^{\omega} = -\frac{e^{\frac{\mathrm{i}\omega|\bx|}{c_T}}}{4\pi\mu|\bx|}\mathbf{I} + \frac{1}{4\pi\omega^2}D^2 \Big(\frac{e^{\frac{\mathrm{i}\omega|\bx|}{c_T}} - e^{\frac{\mathrm{i}\omega|\bx|}{c_L}}}{|\bx|}\Big),
\end{equation}
where $\mathbf{I}$ is the $3\times 3$ identity matrix, $D^2$ denotes the standard double differentiation, and
\begin{equation}\label{eq:definition_ct_cl}
  c_T=\sqrt{\mu} , \quad c_{L}=\sqrt{\lambda + 2 \mu}.
\end{equation}
If $\omega=0$, $\mathbf{\Gamma}^0$ is the Kelvin matrix of the fundamental solution to the elastostatic system, and it is given by
\begin{equation}\label{eq:fundal0}
  \mathbf{\Gamma}^0(\bx)=-\frac{\gamma_1}{4\pi} \frac{1}{|\bx|}\mathbf{I} -\frac{\gamma_2}{4\pi} \frac{\bx \bx^T}{|\bx|^3},
\end{equation}
where the superscript $T$ stands for the transpose of a matrix and
\begin{equation}
  \gamma_1=\frac{1}{2}\left( \frac{1}{\mu} + \frac{1}{2\mu+\lambda} \right) \quad \mbox{and} \quad \gamma_2=\frac{1}{2}\left( \frac{1}{\mu} - \frac{1}{2\mu+\lambda}\right).
\end{equation}
Based on the above notations, the single and double layer potentials can then be defined by
\begin{equation}
  \Scal^{\omega}_{D}[\bvarphi](\bx) = \int_{\partial D} \mathbf{\Gamma}^{\omega}(\bx-\by)\bvarphi(\by)ds(\by), \quad \bx\in\mathbb{R}^3,
\end{equation}
\begin{equation}
  \Dcal^{\omega}_{D}[\bvarphi](\bx) = \int_{\partial D} \frac{\partial}{\partial\bnu_\by} \mathbf{\Gamma}^{\omega}(\bx-\by)\bvarphi(\by)ds(\by), \quad \bx \in\mathbb{R}^3\backslash\partial D,
\end{equation}
for any $\bvarphi\in H^{-1/2}(\partial D)^3$, where and also in what follows, the conormal derivative on $\partial D$ for a function $\bu$ is defined by
\begin{equation}
  \frac{\partial \bu}{\partial\bnu} = \lambda(\nabla \cdot \bu)\bnu + \mu(\nabla \bu + \nabla \bu^T)\bnu,
\end{equation}
with $\bnu$ signifying the exterior unit normal vector to the boundary $\partial D$. There hold the following jump relations for the double layer potential and the conormal derivative of the single layer potential on the boundary (cf. \cite{ABGK}),
\begin{equation}\label{eq:jump_single}
  \frac{\partial}{\partial \bnu} \Scal^{\omega}_D[\bvarphi]|_{\pm}(\bx)=\left( \pm \frac{1}{2}I + (\Kcal^\omega_D)^* \right) [\bvarphi](\bx), \quad \mbox{a.e.} \; \bx\in \partial D,
\end{equation}
\begin{equation}
  \Dcal^{\omega}_D[\bvarphi]|_{\pm}(\bx)=\left( \mp \frac{1}{2}I + \Kcal^\omega_D \right) [\bvarphi](\bx), \quad \mbox{a.e.} \; \bx\in \partial D,
\end{equation}
where $I$ is the identity operator and $\Kcal^{\omega}_D$ is the operator defined by
\begin{equation}
  \Kcal^{\omega}_{D}[\bvarphi](\bx)=\mathrm{p.v.}\int_{\partial D} \frac{\partial}{\partial\bnu_\by} \mathbf{\Gamma}^{\omega}(\bx-\by)\bvarphi(\by)ds(\by),
\end{equation}
and $(\Kcal^{\omega}_D)^*$, also called the Neumann-Poinc\'are (NP) operator, is given as follows
\begin{equation}
  (\Kcal^{\omega}_D)^*[\bvarphi](\bx)=\mathrm{p.v.}\int_{\partial D} \frac{\partial}{\partial\bnu_\bx} \mathbf{\Gamma}^{\omega}(\bx-\by)\bvarphi(\by)ds(\by),
\end{equation}
where p.v. means the Cauchy principal value. In what follows, when $\omega=0$, we denote $\Scal^{0}_{D}, \Dcal^{0}_{D}, \Kcal^{0}_{D}$, and $(\Kcal^{0}_{D})^*$ by $\Scal_{D}, \Dcal_{D}, \Kcal_{D}$, and $\Kcal^*_{D}$ for simplicity.

With the above preparations, one has the following integral representation of the solution to the elastic system \eqnref{eq:def01}-\eqref{eq:radiating} with $(\bvarphi,\bpsi)\in H^{-1/2}(\partial D)^3\times H^{-1/2}(\partial D)^3$ (cf. \cite{ABGK}),
\begin{equation}\label{eq:general_solution}
 \bu(\bx)=
  \left\{
    \begin{array}{ll}
      \Scal^{\omega_1}_D[\bvarphi](\bx), & \bx\in D,\medskip \\
      \Scal^{\omega_2}_D[\bpsi](\bx) + \bF(\bx), &  \bx\in \mathbb{R}^3\backslash \overline{D},
    \end{array}
  \right.
\end{equation}
where
\begin{equation}\label{eq:souce}
  \bF(x) = \int_{\mathbb{R}^3} \bGamma^{\omega_2}(\bx-\by)\bff(\by)d\by,
\end{equation}
\begin{equation}
  \omega_1 = \frac{\omega}{\sqrt{c}} \quad \mbox{with}\quad \Re \omega_1>0, \ \Im \omega_1\leq 0,
\end{equation}
and
\[
 \omega_2 = \omega.
\]
It can be verified that
\begin{equation}\label{eq:LF}
  \mathcal{L}_{\lambda,\mu} \mathbf{F} + \omega^2 \mathbf{F}=\mathbf{f} \quad \mbox{in} \quad \mathbb{R}^3,
\end{equation}
By using the transmission conditions across the boundary $\p D$, the solution $\bu(\bx)$ in \eqref{eq:general_solution} should satisfy
\begin{equation}
  \left\{
    \begin{array}{ll}
      \Scal^{\omega_1}_D[\bvarphi]|_{-} = (\Scal^{\omega_2}_D[\bpsi](\bx) + \bF(\bx))|_{+}, & \bx\in \partial D,\medskip \\
      \frac{\partial}{\partial \bnu}\Scal^{\omega_1}_D[\bvarphi]|_{-} =  \frac{\partial}{\partial \bnu}(\Scal^{\omega_2}_D[\bpsi](\bx) + \bF(\bx))|_{+}, & \bx\in \partial D.
    \end{array}
  \right.
\end{equation}
With the help of the jump relation \eqref{eq:jump_single}, the pair $(\bvarphi, \bpsi) $ is the solution to the following system of integral equations,
\begin{equation}\label{eq:transmission condition}
  \left[
    \begin{array}{cc}
       \Scal^{\omega_1}_D & -\Scal^{\omega_2}_D\medskip \\
      c\left(-\frac{I}{2} + (\Kcal^{\omega_1}_D)^*\right) & -\frac{I}{2} - (\Kcal^{\omega_2}_D)^* \\
    \end{array}
  \right]
 \left[
   \begin{array}{c}
     \bvarphi\medskip \\
     \bpsi \\
   \end{array}
 \right]= \left.
 \left[
   \begin{array}{c}
     \bF\medskip \\
     \frac{\partial\bF}{\partial \bnu} \\
   \end{array}
 \right] \right|_{\partial D}.
\end{equation}

\subsection{Asymptotics for the integral operators}
In this subsection we derive the asymptotic properties for the operators $\Scal^{\omega}_D$ and $(\Kcal^{\omega}_D)^*$ involved in the integral reformulation of the elastic system in Subsection~2.1. We suppose that $\delta\in\mathbb{R}_+$ is sufficiently small, that is, the size of the nanoparticle $D$ is small enough. First of all, we introduce some notations. For a multi-index $\balpha\in\mathbb{N}^3$, let $\bx^{\balpha} = \bx^{\balpha_1}_1 \bx^{\balpha_2}_2 \bx^{\balpha_3}_3$ and $\partial^{\balpha}=\partial^{\balpha_1}_1 \partial^{\balpha_2}_2 \partial^{\balpha_2}_2$, with $\partial_j=\partial/\partial \bx_j$. We also define by $\mathbf{e}_j$, $j=1, 2, 3$ the unit coordinate vectors in $\RR^3$.
Since $D=\delta B + \bz$, for any $\by\in \partial D$, we let $\tilde{\by}=(\by-\bz)/\delta\in \partial B$. Denote by $\widetilde{\bvarphi}(\widetilde{\by})=\bvarphi(\by)$ and $\widetilde{\bpsi}(\widetilde{\by})=\bpsi(\by)$, and let $\partial/\partial\widetilde{\bnu}$ be the conormal derivative on the boundary $\partial B$.

\begin{lem}
Suppose $\delta\in\mathbb{R}_+$ and $\delta\ll 1$, then there holds the following asymptotic expansion for $\bx\in\mathbb{R}^3$
\begin{equation}\label{eq:fund_solution_expansion}
 \begin{split}
  \mathbf{\Gamma}^{\delta}(\bx) = & -\frac{1}{4\pi}\sum_{n=0}^{+\infty}\frac{\mathrm{i}^n}{(n+2)n!}\left(\frac{n+1}{c^{n+2}_T} + \frac{1}{c^{n+2}_L} \right) \delta^n|\bx|^{n-1}\mathbf{I} \\
     & +\frac{1}{4\pi}\sum_{n=0}^{+\infty}\frac{\mathrm{i}^n(n-1)}{(n+2)n!}\left(\frac{1}{c^{n+2}_T} - \frac{1}{c^{n+2}_L} \right) \delta^n |\bx|^{n-3}\bx\bx^T.
 \end{split}
\end{equation}
\end{lem}
\begin{proof}
By using \eqnref{eq:expression_gamma} and Taylor series expansion with respect to $\delta$ and straightforward computations, one can obtain \eqnref{eq:fund_solution_expansion}.

The proof is complete.
\end{proof}
\begin{prop}\label{prop:expansion_single_layer}
  Let $\bvarphi\in H^{-1/2}(\partial D)^3$. For $\delta\in\mathbb{R}_+$ and $\delta\ll 1$, there holds
\begin{equation}
 \Scal^{\omega}_D[\bvarphi](\bx)= \delta \Scal_B[\widetilde{\bvarphi}](\widetilde{\bx}) + \delta^2\omega\Rcal_B[\widetilde{\bvarphi}](\widetilde{\bx}) + \delta^3\omega^2\mathcal{I}_B [\widetilde{\bvarphi}](\widetilde{\bx})+ \mathcal{O}(\delta^4),
\end{equation}
where
\begin{equation}\label{eq:definition_operator_R}
  \Rcal_B[\widetilde{\bvarphi}](\widetilde{\bx})=\gamma_3\int_{\partial B}\widetilde{\bvarphi}(\widetilde{\by})ds(\widetilde{\by}),
\end{equation}
with
\[
\gamma_3=\frac{-\mathrm{i}}{12\pi}\left(\frac{2}{c_{T}^3}+c_{L}^3\right),
\]
and
\begin{equation}\label{eq:definition_operator_I}
  \mathcal{I}_B[\widetilde{\bvarphi}](\widetilde{\bx})=\int_{\partial B}{\mathbf{\Lambda}(\widetilde{\bx}-\widetilde{\by})} \widetilde{\bvarphi}(\widetilde{\by})ds(\widetilde{\by}),
\end{equation}
with
\beq\label{eq:definition_Lambda}
 \mathbf{\Lambda}=\frac{1}{32\pi}\left(\frac{3}{c^{4}_T} + \frac{1}{c^{4}_L} \right)|\bx|\mathbf{I}-\frac{1}{32\pi}\left(\frac{1}{c^{4}_T} - \frac{1}{c^{4}_L} \right)\frac{\bx\bx^T}{|\bx|},
\eeq
and $c_T$ and $c_L$ defined in \eqref{eq:definition_ct_cl}.
\end{prop}
\begin{proof}
  Let $\bx\in\partial D$ and denote $\widetilde{\bx}=(\bx-\bz)/\delta$. Then one has that
  \[
  \Scal^{\omega}_{D}[\bvarphi](\delta \widetilde{\bx}+\bz) = \int_{\partial D} \mathbf{\Gamma}^{\omega}(\delta \widetilde{\bx}+\bz-\by)\bvarphi(\by)ds(\by).
  \]
By using $\by=\delta \widetilde{\by}+\bz$ and change of variables in the above integral, there holds
  \[
  \Scal^{\omega}_{D}[\bvarphi](\delta \widetilde{\bx}+\bz) = \int_{\partial B} \mathbf{\Gamma}^{\omega}(\delta \widetilde{\bx}-\delta\widetilde{\by})\widetilde{\bvarphi}(\widetilde{\by})\delta^2 ds(\tilde{\by}).
  \]
  Substituting \eqref{eq:expression_gamma}, the expression for $\mathbf{\Gamma}^{\omega}$, into the previous equation yields that
  \[
  \Scal^{\omega}_{D}[\bvarphi](\delta \widetilde{\bx}+\bz) = \int_{\partial B}\delta \mathbf{\Gamma}^{\delta\omega}( \widetilde{\bx}-\widetilde{\by})\widetilde{\bvarphi}(\widetilde{\by})ds(\widetilde{\by}).
  \]
  Therefore, the following identity is achieved,
\begin{equation}\label{eq:singdb}
  \Scal^{\omega}_D[\bvarphi](\bx)= \delta \Scal^{\delta \omega}_B[\widetilde{\bvarphi}](\widetilde{\bx}).
\end{equation}
 When $\delta$ is sufficient small, for any $\widetilde{\bx}\in\partial B$, from the expansion \eqref{eq:fund_solution_expansion} for the fundamental solution it follows that
 \[
\Scal^{\omega}_D[\bvarphi](\bx)= \delta \Scal_B[\widetilde{\bvarphi}](\widetilde{\bx}) + \delta^2\omega\Rcal_B[\widetilde{\bvarphi}](\widetilde{\bx}) + \delta^3\omega^2\mathcal{I}_B [\widetilde{\bvarphi}](\widetilde{\bx})+ \mathcal{O}(\delta^4).
 \]

 The proof is complete.
\end{proof}

\begin{prop}\label{prop:expansion_NP_operator}
  Let $\bvarphi\in H^{-1/2}(\partial D)^3$. For $\delta\in\mathbb{R}_+$ and $\delta\ll 1$, there holds
\begin{equation}
  (\Kcal^{\omega}_D)^*[\bvarphi](\bx)= \Kcal^*_B[\widetilde{\bvarphi}](\widetilde{\bx}) + \delta^2\omega^2\mathcal{P}_{B}[\widetilde{\bvarphi}](\widetilde{\bx}) +\mathcal{O}(\delta^3),
\end{equation}
where
\begin{equation}\label{eq:definition_operator_P}
  \mathcal{P}_{B}[\widetilde{\bvarphi}](\widetilde{\bx}) =\int_{\partial B} \frac{\partial}{\partial\bnu_{\widetilde{\bx}}} \mathbf{\Lambda}(\widetilde{\bx}- \widetilde{\by})\widetilde{\bvarphi}(\widetilde{\by}) ds(\widetilde{\by}),
\end{equation}
with $\mathbf{\Lambda}$ defined in \eqref{eq:definition_Lambda}.
\end{prop}
\begin{proof}
Denoting $\widetilde{\bx}=(\bx-\bz)/\delta$ for any $\bx\in\partial D$, one can find that
  \[
  (\Kcal^{\omega}_D)^*[\bvarphi](\delta \widetilde{\bx}+\bz)=\mathrm{p.v.}\frac{1}{\delta}\int_{\partial D} \frac{\partial}{\partial\bnu_{\tilde{\bx}}} \mathbf{\Gamma}^{\omega}(\delta \widetilde{\bx}+\bz-\by)\bvarphi(\by)ds(\by),
  \]
 since for any $\bx\in\partial D$ there holds $\bnu_{D}(\bx)=\bnu_{B}(\widetilde{\bx})$. Then one has by setting $\by=\delta \widetilde{\by}+\bz$ in the previous integral that
 \[
  (\Kcal^{\omega}_D)^*[\bvarphi](\delta \widetilde{\bx}+\bz)=\mathrm{p.v.}\frac{1}{\delta}\int_{\partial B} \frac{\partial}{\partial\bnu_{\widetilde{\bx}}} \mathbf{\Gamma}^{\omega}(\delta \widetilde{\bx}-\delta \widetilde{\by})\widetilde{\bvarphi}(\widetilde{\by})\delta^2 ds(\widetilde{\by}).
 \]
  By substituting \eqref{eq:expression_gamma} into the last equation and direct calculation, one can obtain that
  \[
  (\Kcal^{\omega}_D)^*[\bvarphi](\delta \widetilde{\bx}+\bz)=\mathrm{p.v.}\int_{\partial B} \frac{\partial}{\partial\bnu_{\widetilde{\bx}}} \mathbf{\Gamma}^{\delta\omega}(\widetilde{\bx}- \widetilde{\by})\widetilde{\bvarphi}(\widetilde{\by}) ds(\widetilde{\by}),
  \]
  which gives the the following identity
\begin{equation}\label{eq:npdb}
  (\Kcal^{\omega}_D)^*[\bvarphi](\bx)= (\Kcal^{\delta \omega}_B)^*[\widetilde{\bvarphi}](\widetilde{\bx}).
\end{equation}
  Thus when $\delta$ is sufficient small, from the expansion \eqref{eq:fund_solution_expansion} for the fundamental solution it follows that
  \[
   (\Kcal^{\omega}_D)^*[\bvarphi](\bx)= \Kcal^*_B[\widetilde{\bvarphi}](\widetilde{\bx}) + \delta^2\omega^2\mathcal{P}_{B}[\widetilde{\bvarphi}](\widetilde{\bx}) +\mathcal{O}(\delta^3),
  \]
  and this completes the proof.
\end{proof}
We have derived the asymptotic expansions for the single layer potential operator and the Neumann-Poincar\'e type operator $(\Kcal_D^\omega)^*$. In what follows, we also prove some important Lemmas for our subsequent usage.
\begin{lem}\label{le:imp01}
Let $\mathcal{I}_B$ be defined in \eqnref{eq:definition_operator_I} and if $\delta\in\mathbb{R}_+$ with $\delta\ll 1$, then for $\tilde\bvarphi\in H^{-1/2}(\partial D)^3$, there holds
\beq\label{eq:leimp01}
\mathcal{L}_{\lambda,\mu}\mathcal{I}_B[\tilde\bvarphi](\tdx)=-\Scal_B[\tilde\bvarphi](\tdx), \quad \tdx\in B.
\eeq
\end{lem}
\begin{proof}
Recall that $\mathbf{\Gamma}^{\delta}$ is the fundamental solution to $\mathcal{L}_{\lambda,\mu}+\delta^2$, which shows that
\beq\label{eq:leimp02}
(\mathcal{L}_{\lambda,\mu}+\delta^2)\mathbf{\Gamma}^{\delta}(\tdx-\tdy)=0, \quad \tdx\in B,\ \ \tdy\in \p B.
\eeq
By substituting \eqnref{eq:fund_solution_expansion} into \eqnref{eq:leimp02} and comparing the coefficients of $\delta^2$ one then has
\beq\label{eq:leimp02?}
\mathcal{L}_{\lambda,\mu}\mathbf{\Lambda}(\tdx-\tdy)=-\mathbf{\Gamma}^0(\tdx-\tdy), \quad \tdx\in B,\ \ \tdy\in \p B,
\eeq
where $\mathbf{\Lambda}$ is given in \eqref{eq:definition_Lambda}. By using the definition of $\mathcal{I}_B$ in \eqnref{eq:definition_operator_I}, one finally obtains \eqnref{eq:leimp01}, which completes the proof.
\end{proof}
\begin{lem}\label{prop:property_K_star}
  Suppose that $\Scal_B[\tilde\bvarphi]=C$ holds on $\partial B$, where $C$ is a constant vector and $\tilde\bvarphi\in H^{-1/2}(\partial D)^3$, then
  \begin{equation}\label{eq:prop01}
    \left(-\frac{I}{2} + \Kcal^*_B\right)\Scal^{-1}_B[C]=0.
  \end{equation}
\end{lem}
\begin{proof}
  When $\Scal_B[\tilde\bvarphi]=C$ on $\partial B$, one can easily obtain that $\Scal_B[\tilde\bvarphi]=C$ in the domain $B$. Indeed, let us introduce the Dirichlet boundary value problem
  \[
    \left\{
    \begin{array}{ll}
      \mathcal{L}_{\lambda,\mu} \mathbf{u}(\tilde\bx)=0, & \tilde\bx\in B, \medskip\\
      \bu(\tilde\bx)=C, & \tilde\bx \in \partial B,
    \end{array}
  \right.
  \]
  which has a unique solution (cf. \cite{ABGK}), and $\bu(\tilde\bx)=C$ is the solution. Then by the jump condition \eqref{eq:jump_single} one can obtain that
  \[
  \left(-\frac{I}{2} + \Kcal^*_B\right)[\bvarphi](\tilde\bx)=0, \quad \tilde\bx\in\partial B.
  \]
 Thus the proof is completed since the operator $\Scal_B$ is invertible (cf. \cite{AKKY}).
\end{proof}

We shall also need the following important lemma about the spectrum of the operator $\Kcal_{B}$, which can be found in \cite{AKKY}.
\begin{lem}\label{prop:spectrum_K}
 Assume that $\mathbf{a}\in\mathbb{R}^3$, then $\mathbf{a}$ is an eigenfunction of the operator $\Kcal_{B}$ and the corresponding eigenvalue is $\frac{1}{2}$, i.e.
 \[
   \Kcal_{B}[\mathbf{a}]=\frac{1}{2}\mathbf{a}.
 \]
\end{lem}

\subsection{Far-Field expansion}
In this section, we derive the far-field expansion for the elastic wave field $\bu(\bx)$ to the equation \eqref{eq:general}.

First of all, by Taylor series expansion, $\bF(\by)$ given in \eqref{eq:souce} has the following expansion since $D$ is a nanoparticle
\begin{equation}\label{eq:direxpF01}
  \bF(\by)=\bF(\delta \widetilde{\by} + \bz) = \sum^{+\infty}_{|\bbeta|=0} \frac{1}{\bbeta!} \delta^{|\bbeta|} \widetilde{\by}^{\bbeta} \partial^{\bbeta}\bF(\bz).
\end{equation}
Let $(\widetilde{\bvarphi}_{\bbeta}, \widetilde{\bpsi}_{\bbeta})$ be the solution to the following equation
\begin{equation}\label{eq:general_phi_psi}
  \left[
    \begin{array}{cc}
       \Scal^{\delta \omega_1}_B & - \Scal^{\delta\omega_2}_B\medskip \\
      c\left(-\frac{I}{2} + (\Kcal^{\delta\omega_1}_B)^*\right) & -\frac{I}{2} - (\Kcal^{\delta\omega_2}_B)^* \\
    \end{array}
  \right]
 \left[
   \begin{array}{c}
     \widetilde{\bvarphi}_{\bbeta}\medskip \\
     \widetilde{\bpsi}_{\bbeta} \\
   \end{array}
 \right]=
 \left[
   \begin{array}{c}
     \widetilde{\by}^{\bbeta} \partial^{\bbeta}\bF(\bz)\medskip \\
     \frac{\partial}{\partial \widetilde{\bnu}}\widetilde{\by}^{\bbeta} \partial^{\bbeta}\bF(\bz) \\
   \end{array}
 \right].
\end{equation}
From the identities \eqref{eq:singdb} and \eqref{eq:npdb} for the single layer potential operator and the NP operator, the linearity of the equation \eqref{eq:transmission condition} and together with the help of the following relationship
$$
\frac{\partial}{\partial \bnu}\bF(\by)=\frac{\partial}{\partial \widetilde{\bnu}}\sum^{+\infty}_{|\bbeta|=0} \frac{1}{\bbeta!} \delta^{|\bbeta|-1} \widetilde{\by}^{\bbeta} \partial^{\bbeta}\bF(\bz),
$$
one can conclude that $(\widetilde{\bvarphi}, \widetilde{\bpsi})$ with the following expression
\beq\label{eq:phipsi01}
\widetilde{\bvarphi}=\sum_{|\bbeta|=0}^{+\infty}\delta^{|\bbeta|-1}\widetilde{\bvarphi}_{\bbeta}, \quad
\widetilde{\bpsi}=\sum_{|\bbeta|=0}^{+\infty}\delta^{|\bbeta|-1}\widetilde{\bpsi}_{\bbeta},
\eeq
is the solution to \eqnref{eq:transmission condition}. Therefore from \eqref{eq:general_solution}, we have the following expansion for $\bu(\bx)$ in $\RR^3\setminus\overline{D}$
\begin{equation}\label{eq:farexp01}
  \bu(\bx)=\bF(\bx) + \sum^{+\infty}_{|\balpha|=0} \sum^{+\infty}_{|\bbeta|=0} \delta^{1+|\balpha|+|\bbeta|} \frac{(-1)^{|\balpha|}}{\balpha!\bbeta!}\partial^{\balpha} \bGamma^{\omega}(\bx-\bz)\int_{\partial B} \widetilde{\by}^{\balpha} \widetilde{\bpsi}_{\bbeta}(\widetilde{\by})ds(\widetilde{\by}).
\end{equation}
For $\balpha, \bbeta\in\mathbb{N}^3$, define
\begin{equation}\label{eq:def_M}
  M_{\balpha,\bbeta} := \int_{\partial B} \widetilde{\by}^{\balpha} \widetilde{\bpsi}_{\bbeta}(\widetilde{\by})ds(\widetilde{\by}).
\end{equation}
Then the following lemma holds.
\begin{lem}\label{lem:solution_outside}
Let $\bu$ be the solution to the system \eqnref{eq:def01}-\eqref{eq:radiating}. Then for $\bx\in\mathbb{R}^3\backslash\overline{D}$, one has
\begin{equation}\label{eq:def_M1}
  \bu(\bx)=\bF(\bx) + \sum^{+\infty}_{|\balpha|=0} \sum^{+\infty}_{|\bbeta|=0} \delta^{1+|\balpha|+|\bbeta|} \frac{(-1)^{|\balpha|}}{\balpha!\bbeta!}\partial^{\balpha} \bGamma^{\omega_2}(\bx-\bz) M_{\balpha,\bbeta} .
\end{equation}
\end{lem}
In order to derive the far-field expansion for the displacement field $\bu(\bx)$, we next analyze the term $M_{\balpha,\bbeta}$ more precisely. From the expression for $M_{\balpha,\bbeta}$ in \eqref{eq:def_M}, $\widetilde{\bpsi}_{\bbeta}$ needs to be treated more carefully.

From \eqnref{eq:general_phi_psi} one can readily find out that $\widetilde{\bvarphi}_{\bbeta}$ and $\widetilde{\bpsi}_{\bbeta}$ still depend on $\delta$.  Thus $\widetilde{\bvarphi}_{\bbeta}$ and $\widetilde{\bpsi}_{\bbeta}$ could be further expanded by
\beq\label{eq:phipsiexp02}
\widetilde{\bvarphi}_{\bbeta}=\sum_{n=0}^{+\infty}\delta^n \widetilde{\bvarphi}_{\bbeta,n},\quad
\widetilde{\bpsi}_{\bbeta}=\sum_{n=0}^{+\infty}\delta^n \widetilde{\bpsi}_{\bbeta,n}.
\eeq
Then by using \eqnref{eq:general_phi_psi}, Proposition \ref{prop:expansion_single_layer} and Proposition \ref{prop:expansion_NP_operator}, one can obtain that
\beq\label{eq:bdinteq01}
\Big(\Acal+\delta\mathcal{T}+\delta^2\mathcal{N}\Big)\left[
        \begin{array}{c}
          \widetilde{\bvarphi}_{\beta,0}+\delta\widetilde{\bvarphi}_{\beta,1}+\delta^2\widetilde{\bvarphi}_{\beta,2} \medskip\\
          \widetilde{\bpsi}_{\beta,0}+\delta\widetilde{\bpsi}_{\beta,1}+\delta^2\widetilde{\bpsi}_{\beta,2} \\
        \end{array}
      \right]= \left[
   \begin{array}{c}
     \widetilde{\by}^{\bbeta} \partial^{\bbeta}\bF(\bz)\medskip \\
     \frac{\partial}{\partial \bnu}\widetilde{\by}^{\bbeta} \partial^{\bbeta}\bF(\bz) \\
   \end{array}
 \right]+\mathcal{O}(\delta^3),
\eeq
where $\Acal$, $\mathcal{T}$ and $\mathcal{N}$ are defined by
\begin{equation}\label{eq:definition_operator_A}
\Acal=
 \left[
    \begin{array}{cc}
       \Scal_B & -\Scal_B\medskip \\
      c\left(-\frac{I}{2} + \Kcal_B^*\right) & -\frac{I}{2} - \Kcal_B^* \\
    \end{array}
  \right], \quad \mathcal{T}=\left[
               \begin{array}{cc}
                 \omega_1 \Rcal_{B} & -\omega_2 \Rcal_{B}\medskip \\
                 0 & 0 \\
               \end{array}
             \right],
\end{equation}
and
\beq\label{eq:definition_operator_N}
 \mathcal{N}=\left[
               \begin{array}{cc}
                 \omega_1^2 \mathcal{I}_{B} & -\omega_2^2 \mathcal{I}_{B}\medskip \\
                 c\omega_1^2 \mathcal{P}_{B} & -\omega_2^2 \mathcal{P}_{B} \\
               \end{array}
             \right].
\eeq
The operators $\mathcal{I}_B$ and $\mathcal{P}_B$ are defined in \eqnref{eq:definition_operator_I} and \eqnref{eq:definition_operator_P}, respectively. Since the highest order on $\widetilde{\bvarphi}_{\beta}$ and $\widetilde{\bpsi}_{\beta}$ that we need is of $\delta^2$ order, we neglect the details of the higher-order terms in \eqnref{eq:bdinteq01} and put them all in $\mathcal{O}(\delta^3)$. We mention that the invertibility of the operator $\Acal^{-1}$ can be found in \cite{ABGK}.

\subsection{Asymptotics for the Potential}

From Lemma \ref{lem:solution_outside}, the elastic wave field has the following asymptotic property when $\delta\ll1$,
\begin{equation}\label{eq:mainexp01}
\begin{split}
  \bu(\bx)=&\bF(\bx) +  \delta \left(\bGamma^{\omega_2}(\bx-\bz) M_{0,0}\right)\medskip  \\
                &+\delta^2 \left(\sum_{|\bbeta|=1} \bGamma^{\omega_2}(\bx-\bz) M_{0,\bbeta}-\sum_{|\balpha|=1}  \partial^{\balpha} \bGamma^{\omega_2}(\bx-\bz) M_{\balpha,0}  \right) \\
                &+\delta^3 \left(\sum_{|\bbeta|=2}\frac{1}{\bbeta!} \bGamma^{\omega_2}(\bx-\bz) M_{0,\bbeta}-\sum_{|\balpha|=1}\sum_{|\bbeta|=1}   \partial^{\balpha} \bGamma^{\omega_2}(\bx-\bz) M_{\balpha,\bbeta} \right.\medskip \\
                &+\left.\sum_{|\balpha|=2}\frac{1}{\balpha!} \partial^{\balpha}\bGamma^{\omega_2}(\bx-\bz) M_{\balpha,0} \right) + \mathcal{O}(\delta^4).
\end{split}
\end{equation}
We next analyze \eqnref{eq:mainexp01} term by term. For the sake of simplicity, we define the parameter $\kappa_c$ by
\beq\label{eq:defkappa}
\kappa_c:=\frac{c+1}{2(c-1)}.
\eeq
It is noted that key parts lying in \eqnref{eq:mainexp01} is $M_{\balpha,\bbeta}$ with $|\balpha|\leq2$ and $|\bbeta|\leq2$. From the expression for $M_{\balpha,\bbeta}$ in \eqref{eq:def_M}, we should first derive the asymptotic expression for $\widetilde{\bpsi}_{\bbeta}$, $|\bbeta|\leq 2$.
We present the estimates in the following propositions.
\begin{prop}\label{prop:estimate_solution}
Let $\widetilde{\bvarphi}_{\bbeta, 0}$ and $\widetilde{\bpsi}_{\bbeta,0}$ be defined in \eqnref{eq:phipsi01}. Then one has for $\bbeta\in\mathbb{N}^3$ that
\begin{equation}\label{eq:psi_beta_0}
  \widetilde{\bpsi}_{\bbeta,0}=\frac{1}{(c-1)}\left(\Kcal^*_B -\kappa_c I\right)^{-1}\left[\frac{\partial}{\partial \bnu}\widetilde{\by}^{\bbeta} \partial^{\bbeta}\bF(\bz)- c\left(-\frac{I}{2} + \Kcal^*_B\right)\Scal^{-1}_B\left[\widetilde{\by}^{\bbeta} \partial^{\bbeta} \bF(\bz)\right] \right],
\end{equation}
and
\begin{equation}\label{eq:psi_phi_0}
  \widetilde{\bvarphi}_{\bbeta, 0}=\widetilde{\bpsi}_{\bbeta, 0} + \Scal^{-1}_B\left[\widetilde{\by}^{\bbeta} \partial^{\bbeta} \bF(\bz)\right].
\end{equation}
Furthermore, there holds that $\widetilde{\bpsi}_{0,0}=0$.
\end{prop}
\begin{proof}
 With the help of \eqnref{eq:bdinteq01}, one can obtain that
\[
\left[
   \begin{array}{c}
     \widetilde{\bvarphi}_{\bbeta,0}\medskip \\
     \widetilde{\bpsi}_{\bbeta,0} \\
   \end{array}
 \right]=  \Acal^{-1}
 \left[
   \begin{array}{c}
     \widetilde{y}^{\bbeta} \partial^{\bbeta}\bF(\bz) \medskip\\
     \frac{\partial}{\partial \nu}\widetilde{\by}^{\bbeta} \partial^{\bbeta}\bF(\bz) \\
   \end{array}
 \right],
\]
namely
\beq\label{eq:bdint0101}
 \Scal_B[\widetilde{\bvarphi}_{\bbeta,0}-\widetilde{\bpsi}_{\bbeta,0}]=\widetilde{\by}^{\bbeta} \partial^{\bbeta}\bF(\bz),
\eeq
and
\beq\label{eq:bdint0102}
  c\left(-\frac{I}{2} + \Kcal_B^*\right)[\widetilde{\bvarphi}_{\bbeta,0}] - \left(\frac{I}{2} + \Kcal_B^*\right)[\widetilde{\bpsi}_{\bbeta,0}]= \frac{\partial}{\partial \bnu}\widetilde{\by}^{\bbeta} \partial^{\bbeta}\bF(\bz).
\eeq
Since the operator $\Scal_B$ is invertible from $H^{-1/2}(\partial B)^3$ to $H^{1/2}(\partial B)^3$, by solving \eqnref{eq:bdint0101} and \eqnref{eq:bdint0102}, one thus has \eqnref{eq:psi_beta_0} and \eqnref{eq:psi_phi_0}.
Let $\beta=0$, then \eqnref{eq:psi_beta_0} turns to
\[
  \widetilde{\bpsi}_{0,0}=- \frac{c}{(c-1)}\left(\Kcal^*_B -\kappa_c I\right)^{-1}\left(-\frac{I}{2} + \Kcal^*_B\right)\Scal^{-1}_B\left[\bF(\bz)\right].
\]
Since $\bF(\bz)$ is a constant vector, by using \eqnref{eq:prop01} one thus has  $\widetilde{\bpsi}_{0,0}=0$.

The proof is complete.
\end{proof}
\begin{rem}
In Proposition~\ref{prop:estimate_solution}, we have made use of the invertibility of the operator $\Kcal^*_B -\kappa_c I$. In principle, we need to impose a certain condition on $\kappa_c$, or equivalently on $c$, in order to have the invertibility. So far, we have only assumed that $c\in\mathbb{C}$ and $\Im c\geq 0$. Further condition on $c$ to guarantee the aforesaid invertibility is given in Theorem~\ref{eq:hN} in what follows. In this article, we are mainly interested in the case with $\Re c<0$. In such a case, one can conclude by Theorem~\ref{eq:hN} that if $c$ is such chosen that
\[
 \Im \kappa_c(\kappa_c^2-k_0^2)\neq 0,
\]
where $k_0$ is given in \eqref{eq:hN}, then the operator $\Kcal^*_B -\kappa_c I$ is invertible. The invertibility of $\Kcal^*_B -\kappa_c I$ can hold in more general scenarios and we shall discuss this in more details in Section 3 (cf. Theorem~3.1 and Remark 3.1), and at this stage, we assume the invertibility of $\Kcal^*_B -\kappa_c I$. It is also worth of pointing out that the invertibility of $\Kcal^*_B -\kappa_c I$ clearly also guarantees the well-posedness of the elastic scattering system \eqref{eq:form1}--\eqref{eq:radiating} for plasmonic nanoparticles.
\end{rem}

After obtaining the asymptotic expression for $ \widetilde{\bpsi}_{\bbeta}$, we are in a position to derive the estimate for $M_{\balpha,\bbeta}$ with $|\balpha|\leq2$ and $|\bbeta|\leq2$.
\begin{prop}\label{prop:estimate_solution2}
There holds the following estimate for $M_{\balpha,\bbeta}$ with $\balpha\in\mathbb{N}^3$ and $\bbeta=0$
\begin{equation}\label{eq:estimate_M_alpha_0}
  M_{\balpha,0}= \delta^2  \int_{\partial B}  \widetilde{\by}^{\balpha} \widetilde{\bpsi}_{0,2}(\widetilde{\by})ds(\widetilde{\by}) + \mathcal{O}(\delta^3),
\end{equation}
where
\begin{equation}\label{eq:psi_02}
  \widetilde{\bpsi}_{0,2}=\frac{-c\omega_1^2}{(c-1)}\left(\Kcal^*_B - \kappa_c I\right)^{-1}\left(\mathcal{P}_B- \left(-\frac{I}{2} + \Kcal^*_B\right)\Scal^{-1}_B \mathcal{I}_B \right)\left[ \Scal^{-1}_B\left[ \bF(\bz)\right]\right],
\end{equation}
with $\mathcal{I}_{B}$ and $\mathcal{P}_{B}$ defined in \eqref{eq:definition_operator_I} and \eqref{eq:definition_operator_P} respectively.
\end{prop}

\begin{proof}
Firstly, from the proposition \ref{prop:property_K_star}, one has that
\begin{equation}
  \widetilde{\bpsi}_{0, 0}=0 \quad \mbox{and} \quad \widetilde{\bvarphi}_{0, 0}= \Scal^{-1}_B\left[ \bF(\bz)\right].
\end{equation}
By comparing the coefficients of order $\delta$ in \eqnref{eq:bdinteq01}, one has that
\begin{equation}\label{eq:precise_equation_psi_0}
\Acal \left[
        \begin{array}{c}
          \widetilde{\bvarphi}_{0, 1}\medskip \\
          \widetilde{\bpsi}_{0, 1} \\
        \end{array}
      \right] +
\mathcal{T} \left[
        \begin{array}{c}
          \widetilde{\bvarphi}_{0, 0}\medskip \\
          0 \\
        \end{array}
      \right]=0,
\end{equation}
where the operator $\Acal$ is given in \eqref{eq:definition_operator_A}.
By solving the equation \eqref{eq:precise_equation_psi_0} one derives that
\begin{equation}\label{eq:psi_01}
  \widetilde{\bpsi}_{0, 1}=\frac{c\omega_1}{(c-1)}\left(\Kcal^*_B - \kappa_c I\right)^{-1}\left(\left(-\frac{I}{2} + \Kcal^*_B\right)\Scal^{-1}_B \Rcal_B[ \widetilde{\bvarphi}_{0, 0}] \right).
\end{equation}
with $\Rcal_{B}$ given in \eqref{eq:definition_operator_R}.
Since $\Rcal_B[ \widetilde{\bvarphi}_{0, 0}]$ is a constant, by Lemma \ref{prop:property_K_star} and \eqref{eq:precise_equation_psi_0} one can obtain that
\[
  \widetilde{\bpsi}_{0, 1}=0 \quad \mbox{and} \quad \widetilde{\bvarphi}_{0, 1}=-\omega_1 \Scal^{-1}_B\Rcal_B[ \widetilde{\bvarphi}_{0, 0}].
\]
Furthermore, by comparing the coefficients of order $\delta^2$ in \eqnref{eq:bdinteq01}, one has that
\begin{equation}\label{eq:more_precise_equation_psi_0}
\Acal \left[
        \begin{array}{c}
          \widetilde{\bvarphi}_{0, 2}\medskip \\
          \widetilde{\bpsi}_{0, 2} \\
        \end{array}
      \right] +
\mathcal{T} \left[
        \begin{array}{c}
          \widetilde{\bvarphi}_{0, 1}\medskip \\
          0 \\
        \end{array}
      \right] +
\mathcal{N} \left[
        \begin{array}{c}
          \widetilde{\bvarphi}_{0, 0}\medskip \\
          0 \\
        \end{array}
      \right]=0,
\end{equation}
where $\mathcal{T}$ and $\mathcal{N}$ are defined in \eqnref{eq:definition_operator_A} and \eqnref{eq:definition_operator_N}, respectively.
By solving the above equation one can find that
\begin{equation}
  \widetilde{\bpsi}_{0,2}=\frac{-c\omega_1^2}{(c-1)}\left(\Kcal^*_B - \kappa_c I\right)^{-1}\left(\mathcal{P}_B- \left(-\frac{I}{2} + \Kcal^*_B\right)\Scal^{-1}_B \mathcal{I}_B \right)[\widetilde{\bvarphi}_{0, 0} ],
\end{equation}
with $\mathcal{I}_{B}$ and $\mathcal{P}_{B}$ defined in \eqref{eq:definition_operator_I} and \eqref{eq:definition_operator_P} respectively.
Finally, by the definition of $M_{\balpha,0}$ in \eqnref{eq:def_M}, one has that
\begin{equation}
  M_{\balpha,0}= \delta^2  \int_{\partial B}  \widetilde{\by}^{\balpha} \widetilde{\bpsi}_{0,2}(\widetilde{\by})ds(\widetilde{\by}) + \mathcal{O}(\delta^3),
\end{equation}
and the proof is completed.
\end{proof}

\begin{prop}\label{prop:estimate_solution3}
One has the following estimate
\begin{equation}\label{eq:estimate_M_0_beta}
  \sum_{|\bbeta|=1} M_{0,\bbeta}=\mathcal{O}(\delta^2).
\end{equation}
\end{prop}

\begin{proof}
From the proof of Proposition \ref{prop:estimate_solution}, one has that for $\widetilde{\by}\in\partial B$,
\begin{equation}\label{eq:system_varphi_psi_beta}
  \left\{
    \begin{array}{ll}
      \Scal_B[\widetilde{\bvarphi}_{\bbeta,0}]- \Scal_B[ \widetilde{\bpsi}_{\bbeta,0} ] =  \widetilde{\by}^{\bbeta} \partial^{\bbeta}\bF(\bz),\medskip\\
      c\frac{\partial}{\partial\bnu} \Scal_B[\widetilde{\bvarphi}_{\bbeta,0}]|_- - \frac{\partial}{\partial\bnu} \Scal_B[\widetilde{\bpsi}_{\bbeta,0}]|_+ = \frac{\partial}{\partial \bnu}\widetilde{\by}^{\bbeta} \partial^{\bbeta}\bF(\bz).
    \end{array}
  \right.
\end{equation}
For $|\beta|=1$, one has that
\[
  \int_{\partial B} \frac{\partial}{\partial\bnu} \left(\widetilde{\by}^{\bbeta} \partial^{\bbeta}\bF(\bz)\right) ds(\widetilde{\by}) =0,
\]
which follows from Green's formula and
\[
  \mathcal{L}_{\lambda,\mu} \left(\widetilde{\by}^{\bbeta} \partial^{\bbeta}\bF(\bz)\right)  = 0, \quad \widetilde{\by}\in B.
\]
Using integration by parts, one can similarly show that
\[
 \int_{\partial B}\frac{\partial}{\partial\bnu} \Scal_B[\widetilde{\bvarphi}_{\bbeta,0}]|_-ds(\widetilde{\by}) = 0 \quad \mbox{and} \quad \int_{\partial B} \frac{\partial}{\partial\bnu} \Scal_B[\widetilde{\bpsi}_{\bbeta,0}]|_- ds(\widetilde{\by})=0.
\]
Therefore from the second equation of \eqref{eq:system_varphi_psi_beta}, one can obtain that
\[
 \int_{\partial B} \frac{\partial}{\partial\bnu} \Scal_B[\widetilde{\bpsi}_{\bbeta,0}]|_+ ds(\widetilde{\by})=0.
\]
With the help of the jump relationship \eqref{eq:jump_single} and Lemma \ref{prop:spectrum_K}, one has that
\begin{equation}
  \int_{\partial B} \widetilde{\bpsi}_{\bbeta,0}(\widetilde{\by})ds(\widetilde{\by})=0, \quad \mbox{for} \quad |\bbeta|=1.
\end{equation}
By using \eqnref{eq:bdinteq01} again, one can further obtain, similar to \eqnref{eq:psi_01}, that
\begin{equation}
  \widetilde{\bpsi}_{\bbeta, 1}=\frac{c\omega_1}{(c-1)}\left(\Kcal^*_B - \kappa_c I\right)^{-1}\left(\left(-\frac{I}{2} + \Kcal^*_B\right)\Scal^{-1}_B \Rcal_B[ \widetilde{\bvarphi}_{\bbeta, 0}] \right),
\end{equation}
and by using Lemma \ref{prop:property_K_star} and noting that $\Rcal_B[ \widetilde{\bvarphi}_{\bbeta, 0}]$ is a constant vector, one can in turn derive that
\[
 \widetilde{\bpsi}_{\bbeta, 1}=0, \quad \mbox{for} \quad |\bbeta|=1.
\]
Thus from the definition \eqnref{eq:def_M}, one can obtain that
\begin{equation}
  \sum_{|\bbeta|=1} M_{0,\bbeta}=\mathcal{O}(\delta^2).
\end{equation}
This proof is complete.
\end{proof}

From the last two Propositions \ref{prop:estimate_solution2} and \ref{prop:estimate_solution3}, one can show that the first three terms in \eqnref{eq:mainexp01}  satisfy
\[
 \delta \left(\bGamma^{\omega_2}(\bx-\bz) M_{0,0}\right)  +\delta^2 \left(\sum_{|\bbeta|=1} \bGamma^{\omega_2}(\bx-\bz) M_{0,\bbeta}-\sum_{|\balpha|=1}  \partial^{\balpha} \bGamma^{\omega_2}(\bx-\bz) M_{\balpha,0} \right)=\mathcal{O}(\delta^4).
\]
Therefore from \eqnref{eq:mainexp01}, the far-field expansion for the displacement $\bu(\bx)$ can be written as
\begin{equation}\label{eq:final_scattering}
 \begin{split}
 \bu(\bx)-\bF(\bx) =& \delta^{3}  \Big(\bGamma^{\omega_2}(\bx-\bz) \int_{\partial B} \widetilde{\bpsi}_{0, 2}(\widetilde{\by})ds(\widetilde{\by}) +\sum_{|\bbeta|=2}
\frac{1}{2} \bGamma^{\omega_2}(\bx-\bz)\int_{\partial B} \widetilde{\bpsi}_{\bbeta,0}(\widetilde{\by})ds(\widetilde{\by})\\
 &- \sum_{|\balpha|=1} \sum_{|\bbeta|=1}\partial^{\balpha} \bGamma^{\omega_2}(\bx-\bz)\int_{\partial B} \widetilde{\by}^{\balpha} \widetilde{\bpsi}_{\bbeta,0}(\widetilde{\by})ds(\widetilde{\by}) \Big)
     +\mathcal{O}(\delta^4),
 \end{split}
\end{equation}
with $\widetilde{\bpsi}_{\bbeta,0}$ and $ \widetilde{\bpsi}_{0, 2}$ given in \eqref{eq:psi_beta_0} and  \eqref{eq:psi_02}, respectively.
The following lemmas show the detailed calculation of the three terms appeared in the right hand side of \eqref{eq:final_scattering}.

\begin{lem}\label{le:asym0101}
There holds the following identity
\begin{equation}\label{eq:estmate1st02}
\sum_{|\bbeta|=2} \int_{\partial B} \widetilde{\bpsi}_{\bbeta,0}(\widetilde{\by})ds(\widetilde{\by})=-2|B|  (\mathcal{L}_{\lambda,\mu} \bF)(\bz).
\end{equation}
\end{lem}
\begin{proof}
Recalling the expression for $\widetilde{\bpsi}_{\bbeta,0}$ given in \eqref{eq:psi_beta_0}, one has that
\begin{equation}\label{eq:der1}
 (c-1)\left(\Kcal^*_B - \kappa_c I\right)[\widetilde{\bpsi}_{\bbeta,0}]=\frac{\partial}{\partial \bnu}\widetilde{\by}^{\bbeta} \partial^{\bbeta}\bF(\bz)- c\left(-\frac{1}{2} + \Kcal^*_B\right)\Scal^{-1}_B\left[\widetilde{\by}^{\bbeta} \partial^{\bbeta} \bF(\bz)\right].
\end{equation}
By using Lemma \ref{prop:spectrum_K}, one can obtain that for $\bvarphi\in H^{-1/2}(\partial B)^3$,
\[
\begin{split}
 &\int_{\partial B}\Kcal^*_B[\bvarphi] ds =  \sum_{i=1}^3 \be_i \int_{\partial B}(\Kcal^*_B[\bvarphi])\cdot \be_i ds\\
 = &\sum_{i=1}^3 \be_i \int_{\partial B}\bvarphi \cdot \Kcal_B[\be_i]  =  \sum_{i=1}^3 \frac{\be_i}{2} \int_{\partial B}\bvarphi\cdot \be_i ds = \frac{1}{2} \int_{\partial B}\bvarphi ds.
\end{split}
\]
Therefore
\begin{equation}\label{eq:der2}
 \int_{\partial B} (c-1)\left(\Kcal^*_B - \kappa_c I\right)[\widetilde{\bpsi}_{\bbeta,0}] ds=-\int_{\partial B} \widetilde{\bpsi}_{\bbeta,0} ds,
\end{equation}
and
\begin{equation}\label{eq:der3}
 \int_{\partial B} \left(-\frac{1}{2} + \Kcal^*_B\right)\Scal^{-1}_B\left[\widetilde{\by}^{\bbeta} \partial^{\bbeta} \partial^{\bbeta} \bF(\bz)\right]ds =0.
\end{equation}
With the help of Green's formula, there holds the following
\begin{equation}\label{eq:der4}
\sum_{|\bbeta|=2} \int_{\partial B} \frac{\partial}{\partial \bnu}\widetilde{\by}^{\bbeta} \partial^{\bbeta}\bF(\bz) ds = \sum_{|\bbeta|=2} \int_{ B} \mathcal{L}_{\lambda,\mu}\widetilde{\by}^{\bbeta} \partial^{\bbeta}\bF(\bz) dv =2|B|  (\mathcal{L}_{\lambda,\mu} \bF)(\bz),
\end{equation}
where $|B|$ is the volume of the domain $B$. Finally, from \eqref{eq:der1}, \eqref{eq:der2}, \eqref{eq:der3} and \eqref{eq:der4}, one can thus obtain \eqnref{eq:estmate1st02}, which completes the proof.
\end{proof}

\begin{lem}\label{le:asym0102}
There holds the following identity
\beq\label{eq:estmate1st01}
\int_{\partial B} \widetilde{\bpsi}_{0, 2}(\widetilde{\by})ds(\widetilde{\by})=-\omega^2|B|\bF(\bz).
\eeq
\end{lem}
\begin{proof}
From \eqnref{eq:psi_02} one has
\begin{equation}\label{eq:der101}
 (c-1)\left(\Kcal^*_B - \kappa_c I\right)[\widetilde{\bpsi}_{0,2}]=-c\omega_1^2\left(\mathcal{P}_B- \left(-\frac{I}{2} + \Kcal^*_B\right)\Scal^{-1}_B \mathcal{I}_B \right)\left[ \Scal^{-1}_B\left[ \bF(\bz)\right]\right].
\end{equation}
Similar to the proof in Lemma \ref{le:asym0101}, one has
\beq\label{eq:der102}
\int_{\partial B} \widetilde{\bpsi}_{0, 2}(\widetilde{\by})ds(\widetilde{\by})=c\omega_1^2\int_{\p B} \mathcal{P}_B\left[ \Scal^{-1}_B\left[ \bF(\bz)\right]\right](\widetilde{\by})ds(\widetilde{\by}).
\eeq
By using the definition of $\mathcal{P}_B$ in \eqnref{eq:definition_operator_P}, integration by parts and \eqnref{eq:leimp01} one obtains
\beq\label{eq:der103}
\begin{split}
\int_{\p B} \mathcal{P}_B\left[ \Scal^{-1}_B\left[ \bF(\bz)\right]\right](\widetilde{\by})ds(\widetilde{\by})&=\int_{\p B}\frac{\p}{\p \bnu}\mathcal{I}_B\left[ \Scal^{-1}_B\left[ \bF(\bz)\right]\right](\widetilde{\by})ds(\widetilde{\by})\\
&=\int_{B}(\mathcal{L}_{\lambda,\mu}\mathcal{I}_B)\left[ \Scal^{-1}_B\left[ \bF(\bz)\right]\right](\widetilde{\by})d\widetilde{\by}\\
&=-\int_{B}\mathcal{S}_B\left[ \Scal^{-1}_B\left[ \bF(\bz)\right]\right](\widetilde{\by})d\widetilde{\by}=- |B|\bF(\bz).
\end{split}
\eeq
By substituting \eqnref{eq:der103} into \eqnref{eq:der102} one thus has \eqnref{eq:estmate1st01}.

The proof is complete.
\end{proof}

\begin{lem}\label{le:axibeta1}
Let $\widetilde{\bpsi}_{\bbeta, 0}$ be defined in \eqnref{eq:phipsi01}. Then for $|\bbeta|=1$, one has
\begin{equation}\label{eq:psibeta1}
  \widetilde{\bpsi}_{\bbeta,0}=-\left(\Kcal^*_B-\kappa_c I\right)^{-1}\Big[\frac{\partial}{\partial \bnu}\widetilde{\by}^{\bbeta} \partial^{\bbeta}\bF(\bz)\Big].
\end{equation}
\end{lem}
\begin{proof}
Recall the formula for $\widetilde{\bpsi}_{\bbeta,0}$ in \eqnref{eq:psi_beta_0}. For $|\bbeta|=1$, it can be easily verified that
\beq\label{eq:axitmp01}
\mathcal{L}_{\lambda,\mu}\big(\widetilde{\by}^{\bbeta} \partial^{\bbeta}\bF(\bz)\big)=0 \quad \mbox{in} \quad \RR^3.
\eeq
Define
$$
\bvarphi:=\Scal^{-1}_B\left[\widetilde{\by}^{\bbeta} \partial^{\bbeta} \bF(\bz)\right] \quad \mbox{on} \quad \p B.
$$
Then one has
\beq\label{eq:axitmp02}
\Scal_B[\bvarphi]=\widetilde{\by}^{\bbeta} \partial^{\bbeta} \bF(\bz) \quad \mbox{on} \quad \p B.
\eeq
Together with \eqnref{eq:axitmp01} one can show that
\beq\label{eq:axitmp03}
\Scal_B[\bvarphi]=\widetilde{\by}^{\bbeta} \partial^{\bbeta} \bF(\bz) \quad \mbox{in} \quad B.
\eeq
Thus one has from \eqnref{eq:jump_single} and \eqnref{eq:axitmp03} that
\beq\label{eq:axitmp04}
\begin{split}
\left(-\frac{I}{2} + \Kcal^*_B\right)\Scal^{-1}_B\left[\widetilde{\by}^{\bbeta} \partial^{\bbeta} \bF(\bz)\right]&=\left(-\frac{I}{2} + \Kcal^*_B\right)[\bvarphi]\\
&=\frac{\partial}{\partial \bnu} \Scal_B[\bvarphi]|_{-}(\bx)
=\frac{\partial}{\partial \bnu} \widetilde{\by}^{\bbeta} \partial^{\bbeta} \bF(\bz).
\end{split}
\eeq
Together with \eqnref{eq:psi_beta_0}, one finally has \eqnref{eq:psibeta1}.

The proof is complete.
\end{proof}
Based on the above results, we can finally show the far-field expansion for the displacement $\bu(\bx)$ as follows
\begin{thm}\label{th:farfexp01}
Let $\bu(\bx)$ be the solution to the system \eqnref{eq:def01}-\eqref{eq:radiating}. Then for $\bx\in\mathbb{R}^3\backslash\overline{D}$, there holds
\begin{equation}\label{eq:farfexp0101}
   \bu(\bx) =\bF(\bx) +\delta^3\sum_{j=1}^3 \sum_{|\balpha|=1} \sum_{|\bbeta|=1} \partial^{\balpha} \bGamma^{\omega}(\bx-\bz) \partial^{\bbeta}\bF_j(\bz) \mathbf{M}_{\alpha,\beta}^j+\mathcal{O}(\delta^4),
\end{equation}
where the Elastic Moment Tensor (EMT) $(\mathbf{M}_{\alpha,\beta}^j)$ is defined by
\beq\label{eq:defPolar01}
\mathbf{M}_{\alpha,\beta}^j=\int_{\partial B} \widetilde{\by}^{\balpha} \left(\Kcal^*_B-\kappa_c I\right)^{-1}\left[\frac{\partial}{\partial \bnu}\widetilde{\by}^{\bbeta} \be_j\right]ds(\widetilde{\by}),
\eeq
and
\[
  \partial^{\bbeta}\bF(\bz) =(\partial^{\bbeta}\bF_1(\bz), \partial^{\bbeta}\bF_2(\bz), \partial^{\bbeta}\bF_3(\bz)).
\]
\end{thm}
\begin{proof}
With the help of Lemmas \ref{le:asym0101} and \ref{le:asym0102}, one can obtain that
\[
\begin{split}
 &\int_{\partial B} \widetilde{\bpsi}_{0, 2}(\widetilde{\by})ds(\widetilde{\by}) +\sum_{|\bbeta|=2} \frac{1}{2} \int_{\partial B} \widetilde{\bpsi}_{\bbeta,0}(\widetilde{\by})ds(\widetilde{\by}) \\
=&|B|\big(-\omega^2 \mathbf{F}(\bz)- (\mathcal{L}_{\lambda,\mu}\mathbf{F})(\bz)\big) =0,
 \end{split}
\]
where the second identity follows from  \eqnref{eq:LF} and the assumption that $\mathbf{f}$ is compactly supported in $\RR^3\setminus\overline{D}$. Therefore one can derive \eqref{eq:farfexp0101} thanks to \eqref{eq:final_scattering} and the linearity of the operator $\left(\Kcal^*_B-\kappa_c I\right)^{-1}$.

This proof is complete.
\end{proof}
We remark that in \cite{AKNT:02} (see also \cite{Amm3}), the authors proved the asymptotic expansion for static elastic problem in the presence of small inclusions. Our asymptotic expansion result \eqnref{eq:farfexp0101} is in accordance with their results but has more details in describing the EMT in \eqnref{eq:defPolar01}. This helps us to analyze the EMT in a much more elaborate way and so the phenomenon of polariton resonance.

\section{Mathematical analysis of surface polariton resonance}
\subsection{Resolvent analysis}
From the analysis of the last section, we know that the polariton resonance only occurs when the energy of the EMT, $(\mathbf{M}_{\alpha,\beta}^j)$, blows-up (see \eqnref{eq:farfexp0101} and \cite{ADM}). From \eqnref{eq:defPolar01}, it then remains to analyze the resolvent of the Neumann-Poincar\'e operator $\Kcal^*_B$. We have the following auxiliary result

\begin{thm}\label{th:resolv01}
For the operator $\Kcal_B^*$, we have the following resolvent estimate
\begin{equation}\label{eq:reso_Kstar}
\|\left(k_{\alpha} I-\Kcal_B^*\right)^{-1}\|_{\mathcal{L}(H^{-1/2}(\partial B)^3,H^{-1/2}(\partial B)^3)}\leq  \frac{C}{d(h(k_{\alpha}),\sigma(\mathcal{G}_B))},
\end{equation}
where
\begin{equation}\label{eq:hN}
 h(k_{\alpha})=k_{\alpha}(k_{\alpha}^2-k_0^2), \quad\mbox{and}\quad  \mathcal{G}_B=\Kcal_B^*((\Kcal_B^*)^2-k_0^2I),
\end{equation}
with
\[
k_0=\frac{\mu}{2(2\mu+\lambda)},
\]
and the constant $C$ depends on the $\Kcal_B^*$, $k_{\alpha}$ and $k_0$. In \eqref{eq:reso_Kstar}, $\sigma(\mathcal{G}_B)$ is the spectrum of the operator $\mathcal{G}_B$ and $d(h(k_{\alpha}),\sigma(\mathcal{G}_B))$ is the distance between $h(k_{\alpha})$ and $\sigma(\mathcal{G}_B)$.
\end{thm}
\begin{proof}
First we introduce the Hilbert space $\mathcal{H}$ with the following inner product
\beq\label{eq:innerdef01}
(\mathbf{g},\mathbf{h})_{\mathcal{H}}:=-\langle\mathbf{g},\Scal_B[\mathbf{h}]\rangle, \quad \mathbf{g},\mathbf{h}\in H^{-1/2}(\partial B)^3,
\eeq
where $\langle.,.\rangle$ is the pair between $H^{-1/2}(\partial B)^3$ and $H^{1/2}(\partial B)^3$, and the operator $\Scal_B$ is defined on $\partial B$.  Then the operator $\Kcal_B^*$ is a self-adjoint operator in the space $\mathcal{H}$.  Furthermore, since the norm $\|\cdot\|_{\mathcal{H}}$ induced by the inner product $(.,.)_{\mathcal{H}}$ is equivalent to $\|\cdot\|_{H^{-1/2}(\partial B)^3}$, we thus denote the norm $\|\cdot\|_{\mathcal{H}}$ by $\|\cdot\|_{H^{-1/2}(\partial B)^3}$ without any ambiguity. We refer to \cite{AKKY} for more details.

One can find that the operator $\mathcal{G}_B$ defined in \eqref{eq:hN} is a compact operator on $\mathcal{H}$ (see \cite{AKM1}). By using the  Calder\'on identity
$$
\Scal_B\Kcal_B^*=\Kcal_B\Scal_B
$$
one can find that
\beq
\begin{split}
(\mathcal{G}_B[\mathbf{g}],\mathbf{h})_{\mathcal{H}}&=-\langle\mathcal{G}_B[\mathbf{g}],\Scal_B[\mathbf{h}]\rangle
=-\langle\mathbf{g}, \mathcal{G}_B^*\Scal_B[\mathbf{h}]\rangle\medskip\\
&= -\langle\mathbf{g}, \Scal_B\mathcal{G}_B[\mathbf{h}]\rangle=(\mathbf{g},\mathcal{G}_B[\mathbf{h}])_{\mathcal{H}},
\end{split}
\eeq
where $\mathcal{G}_B^*:=\Kcal_B((\Kcal_B)^2-k_0^2I)$.
Thus it is a self-adjoint compact operator on $\mathcal{H}$ and therefore from \cite{GM} one can derive the resolvent for $\mathcal{G}_B$ by
\[
 \|\left(h(k_{\alpha}) I-\mathcal{G}_B\right)^{-1}\|_{\mathcal{L}(H^{-1/2}(\partial B)^3,H^{-1/2}(\partial B)^3)}\leq \frac{C}{d(h(k_{\alpha}),\sigma(\mathcal{G}_B))}.
\]
Direct calculation gives that
\[
 \left(\Kcal_B^*-k_{\alpha} I\right)\left((\Kcal_B^*)^2+k_{\alpha} \Kcal_B^*+(k_{\alpha}^2-k_0^2) I\right)=h(k_{\alpha}) I-\mathcal{G}_B,
\]
and therefore
\beq\label{eq:thadd0101}
 \left(\Kcal_B^*-k_{\alpha} I\right)^{-1}=\left(h(k_{\alpha}) I-\mathcal{G}_B\right)^{-1} \left((\Kcal_B^*)^2+k_{\alpha} \Kcal_B^*+ (k_{\alpha}^2-k_0^2) I\right).
\eeq
Finally one can conclude that
\[
 \|\left(k_{\alpha} I-\Kcal_B^*\right)^{-1}\|_{\mathcal{L}(H^{-1/2}(\partial B)^3,H^{-1/2}(\partial B)^3)}\leq \frac{C}{d(h(k_{\alpha}),\sigma(\mathcal{G}_B))},
\]
and the proof is completed.
\end{proof}

From Theorems \ref{th:resolv01} and \ref{th:farfexp01}, one can find that the spectra of the Neumann-Poincar\'e operator $\Kcal_B^*$ are attributed to the polariton resonance of the elastic problem, which we shall discuss in more details in the next subsection.

\subsection{Polariton resonance for elastic nanoparticles}

\begin{thm}\label{th:main02}
Let $\bu(\bx)$ be the solution to the system \eqnref{eq:def01}-\eqref{eq:radiating}. Suppose that $c=c_0+ i\tau$, where $\tau\in\mathbb{R}_+$ is sufficiently small and $c_0$ is chosen such that the quantity $h_{c_0}$
$$h_{c_0}:=k_{c_0}(k_{c_0}^2-k_0^2),$$
with
$$ k_{c_0}=\frac{c_0+1}{2(c_0-1)}\quad \mbox{and} \quad k_0=\frac{\mu}{2(2\mu+\lambda)},$$
is the eigenvalue of the operator $\mathcal{G}_B$ with respect to an eigenfunction $\bphi^*$. Suppose further that
\beq\label{eq:thmain02021}
(\bphi^*, \bphi_F)_{\mathcal{H}}\neq 0,
\eeq
where
$$
\bphi_F(\widetilde{\by})=\left((\Kcal_B^*)^2+k_c \Kcal_B^*+ (k_c^2-k_0^2) I\right)\Big[\Big(\frac{\partial}{\partial \bnu}\bF\Big)(\bz)\Big](\widetilde{\by}),
$$
and
$$
\bphi_G(\bx, \widetilde{\by})=\sum_{|\alpha|=1}\p^{\alpha}  \bGamma^{\omega}(\bx-\bz)\widetilde{\by}^{\balpha}.
$$
Then one has
\beq\label{eq:thmain0201}
|\bu(\bx)|\sim \delta^3\tau^{-1}, \quad \mbox{a.e.} \quad \bx\in \RR^3\setminus \overline{D}.
\eeq
\end{thm}
\begin{proof}
Firstly, one can rewrite the asymptotic expansion \eqnref{eq:farfexp0101} as follows
\beq\label{eq:thmain020101}
\bu(\bx) =\bF(\bx) +\delta^3\mathbf{P}(\bx)+\mathcal{O}(\delta^4),
\eeq
where $\mathbf{P}$ is defined by
\beq\label{eq:thmain0202}
\mathbf{P}(\bx):=\int_{\partial B} \bphi_G(\bx, \widetilde{\by})\left(\Kcal^*_B-\kappa_c I\right)^{-1}\left[\Big(\frac{\partial}{\partial \bnu}\bF\Big)(\bz)\right]ds(\widetilde{\by}).
\eeq
Recall that $\mathcal{G}_B$ defined in \eqnref{eq:hN} is a self-adjoint compact operator on $\mathcal{H}$, where $\mathcal{H}$ is a Hilbert space with functions in $H^{-1/2}(\partial B)^3$ and inner product defined by \eqnref{eq:innerdef01}. Thus the eigenfunctions $\{\bphi_n\}_{n=1}^{\infty}$, corresponding the eigenvalues $\{\lambda_n\}_{n=1}^{\infty}$ of the operator $\mathcal{G}_B$, form a norm basis on $\mathcal{H}$ and therefore the operator $\mathcal{G}_B$ admits the follwoing eigenvalue decompositions in $\mathcal{H}$, for $\bphi\in H^{-1/2}(\partial B)^3$
\beq\label{eq:thmain0203}
\mathcal{G}_B[\bphi]=\sum_{n=1}^{\infty} \kappa_n \bphi_n,
\eeq
where
\[
 \kappa_n=\lambda_n (\bphi_n,\bphi)_{\mathcal{H}}.
\]
Since $c=c_0 +i \tau$, one has
$$
\kappa_c=\frac{c+1}{2(c-1)}=k_c + \mathcal{O}(\tau).
$$
By using the relations in \eqnref{eq:thadd0101}, \eqnref{eq:thmain02021} and \eqnref{eq:thmain0203}, one then has
\beq\label{eq:thmain0204}
\begin{split}
\left(\Kcal^*_B-\kappa_c I\right)^{-1}\left[\Big(\frac{\partial}{\partial \bnu}\bF\Big)(\bz)\right]&=\left(h_c I-\mathcal{G}_B+\mathcal{O}(\tau)\right)^{-1}[\bphi_F]\\
&=C\left(h_c I-\mathcal{G}_B+\mathcal{O}(\tau)\right)^{-1}[\bphi^*] + \mathcal{O}(1)\\
&=\bphi^*\mathcal{O}(\tau^{-1})+\mathcal{O}(1).
\end{split}
\eeq
Since there holds
\[
(\mathcal{L}_{\lambda, \mu}+\omega^2) \mathbf{P}(\bx) =0 ,\quad \bx\in\RR^3\setminus\overline{D},
\]
by using the unique continuation principle, one has that
\[
\mathbf{P}(\bx)\neq 0, \quad \mbox{a.e.} \quad \bx \in\RR^3\setminus\overline{D}.
\]
By substituting \eqnref{eq:thmain0204} into \eqnref{eq:thmain0202} one thus has
$$
|\mathbf{P}(\bx)|\sim \tau^{-1}, \quad \mbox{a.e.} \quad \bx\in \RR^3\setminus\overline{D},
$$
which together with \eqnref{eq:thmain020101} readily implies \eqnref{eq:thmain0201}.

The proof is complete.
\end{proof}

\begin{rem}\label{rem:3.1}
We remark that if the polariton resonance occurs, the parameter $c$ defined in \eqref{eq:def02} should have negative real part, namely
\[
 \Re c=c_0<0.
\]
In fact, if $\Re c>0$, one can show from the definition of $\kappa_c$ in \eqref{eq:defkappa} that
\[
 \Re \kappa_c>\frac{1}{2}.
\]
Then in such a case, one cannot have the polarition resonance since
\[
  \|\left(\kappa_{c} I-\Kcal_B^*\right)^{-1}\|_{\mathcal{L}(H^{-1/2}(\partial B)^3,H^{-1/2}(\partial B)^3)}\leq C,
\]
which is due to the fact that the spectrum of the operator $\Kcal_B^*$ on $H^{-1/2}(\partial B)^3$ lies in $(-1/2, 1/2]$ (cf. \cite{AKKY}).
\end{rem}

\section*{Acknowledgment}
The work of Y. Deng was supported by NSF grant of China No. 11601528, NSF grant of Hunan No. 2017JJ3432 and No. 2018JJ3622, Innovation-Driven Project of Central South University, No. 2018CX041. The work of H. Liu was supported by the startup fund and FRG grants from Hong Kong Baptist University, Hong Kong RGC General Research Funds, 12302415 and 12302017.

%
%

\end{document}